\documentclass[dvipsnames,version=final]{amsart}

\usepackage{todonotes}
\usepackage[T1]{fontenc}
\usepackage[utf8]{inputenc}
\usepackage[english]{babel}
\usepackage{geometry}
\usepackage{amsmath}
\usepackage{amsfonts}
\usepackage{amssymb}
\usepackage{booktabs}
\usepackage{paralist}
\usepackage{array}
\usepackage{xy}
\usepackage{multicol}
\usepackage{multirow}
\usepackage{fancyhdr}
\usepackage{hyperref}
\usepackage{wrapfig}
\usepackage{bm}
\usepackage{listings}
\usepackage{enumerate}
\usepackage{multirow}
\usepackage{placeins}
\usepackage{pgf}
\usepackage{tikz}
\usepackage{mathtools}
\usepackage{algorithm}
\usepackage{algpseudocode}
\usepackage{dsfont}
\usepackage{bbm}
\usepackage{subcaption}
\usepackage{cleveref}

\DeclareMathOperator{\Off}{Off}
\DeclareMathOperator{\trace}{tr} 	% matrix trace
\DeclareMathOperator{\dist}{d} 		% distance
\DeclareMathOperator{\col}{col}
\DeclareMathOperator{\traceP}{\mathcal{T}}

\DeclareMathOperator{\Li}{Li}
\DeclareMathOperator{\erf}{erf}
\DeclareMathOperator{\inverf}{inverf}

\newcommand{\poly}{\Pi}				% space of polynomials

\newcommand{\R}{\mathbb{R}}

\newcommand{\Z}{\mathbb{Z}}%
\newcommand{\N}{\mathcal{N}}
\newcommand{\prob}{\mathbb{P}}	% probability
\renewcommand{\O}{\mathcal{O}}
\newcommand{\graph}{\mathcal{G}}
\renewcommand{\vec}[1]{\boldsymbol{#1}}		% for bold vectors

\newcommand{\pp}{\nolinebreak\hspace{-.05em}\raisebox{.4ex}{\tiny\bf +}\nolinebreak\hspace{-.10em}\raisebox{.4ex}{\tiny\bf +}}

\newcommand{\E}{\mathbb{E}}
\newcommand{\Var}{\mathbb{V}}
\newcommand{\V}{V}

\usepackage{pgfplots}
\usepgfplotslibrary{groupplots}
\usetikzlibrary{external}

\usetikzlibrary{decorations.markings}
\makeatletter
\tikzset{
  nomorepostactions/.code={\let\tikz@postactions=\pgfutil@empty},
  mymark/.style 2 args={decoration={markings,
    mark= between positions 0 and 1 step (1/9)*\pgfdecoratedpathlength with{%
        \tikzset{#2,every mark}\tikz@options
        \pgfuseplotmark{#1}%
      },  
    },
    postaction={decorate},
    /pgfplots/legend image post style={
        mark=#1,mark options={#2},every path/.append style={nomorepostactions}
    },
  },
}
\makeatother

\pgfplotscreateplotcyclelist{list_scaling_n}{%
Sepia,line width=1.25pt, mark=square, solid, mark size=2.5pt\\
BurntOrange,line width=1.25pt, mark=o, solid, mark size=2.5pt\\
OliveGreen,line width=1.25pt, mark=diamond, solid, mark size=2.5pt\\
Cerulean,line width=1.25pt, mark=x, dashed, mark size=2.5pt,mark options={solid}\\
Orchid,line width=1.25pt, mark=triangle, dashed, mark size=2.5pt,mark options={solid}\\
}

\pgfplotscreateplotcyclelist{list_matvecs}{%
Cerulean,line width=1.25pt, mark=x, dashed, mark size=2.5pt,mark options={solid}\\
Orchid,line width=1.25pt, mark=triangle, dashed, mark size=2.5pt,mark options={solid}\\
WildStrawberry,line width=1.25pt, mark=pentagon, dash dot, mark size=2.5pt,mark options={solid}\\
SpringGreen,line width=1.25pt, mark=Mercedes star, dash dot, mark size=2.5pt,mark options={solid}\\
Sepia,line width=1.25pt, mark=square, solid, mark size=2.5pt\\
BurntOrange,line width=1.25pt, mark=o, solid, mark size=2.5pt\\
OliveGreen,line width=1.25pt, mark=diamond, solid, mark size=2.5pt\\
}

\pgfplotscreateplotcyclelist{list_eig}{%
JungleGreen, line width=1.25pt, solid\\
Magenta, line width=1.25pt, dashed\\
Violet, line width=1.25pt, dash dot\\
}

\pgfplotsset{compat=1.16}
\usepgfplotslibrary{groupplots}

\usepgfplotslibrary{external} 
\newtheorem{theorem}{Theorem}[section]
\newtheorem{lemma}[theorem]{Lemma}
\newtheorem{proposition}[theorem]{Proposition}

\theoremstyle{definition}
\newtheorem{definition}[theorem]{Definition}

\theoremstyle{remark}
\newtheorem{remark}[theorem]{Remark}

\numberwithin{equation}{section}

\begin{document}

\title[Analysis of stochastic probing]{Analysis of stochastic probing methods for estimating the trace of functions of sparse symmetric matrices}

\author[A.~Frommer]{Andreas Frommer}
\address{School of Mathematics and Natural Sciences, Bergische Universit\"at Wuppertal, 42097 Wuppertal, Germany}
\email{frommer@uni-wuppertal.de}

\author[M.~Rinelli]{Michele Rinelli}
\address{Scuola Normale Superiore, Piazza dei Cavalieri, 7, 56126 Pisa, Italy}
\email{michele.rinelli@sns.it}

\author[M.~Schweitzer]{Marcel Schweitzer}
\address{School of Mathematics and Natural Sciences, Bergische Universit\"at Wuppertal, 42097 Wuppertal, Germany}
\email{marcel@uni-wuppertal.de}


\date{}

\dedicatory{}


\begin{abstract}
We consider the problem of estimating the trace of a matrix function $f(A)$. In certain situations, in particular if $f(A)$ cannot be well approximated by a low-rank matrix, combining probing methods based on graph colorings with stochastic trace estimation techniques can yield accurate approximations at moderate cost. So far, such methods have not been thoroughly analyzed, though, but were rather used as efficient heuristics by practitioners. In this manuscript, we perform a detailed analysis of stochastic probing methods and, in particular, expose conditions under which the expected approximation error in the stochastic probing method scales more favorably with the dimension of the matrix than the error in non-stochastic probing. Extending results from [E.\ Aune, D.\ P.\ Simpson, J.\ Eidsvik, Parameter estimation in high dimensional Gaussian distributions, Stat.\ Comput., 24, pp.\ 247--263, 2014], we also characterize situations in which using just one stochastic vector is always---not only in expectation---better than the deterministic probing method. Several numerical experiments illustrate our theory and compare with existing methods. 
\end{abstract}

\maketitle

\section{Introduction}
Estimating the trace of an implicitly given matrix $B \in \R^{n \times n}$, 
\begin{equation}\label{eq:trB}
\trace(B) = \sum\limits_{i=1}^n [B]_{ii},
\end{equation}
is an important task in many areas of applied mathematics and computer science, with applications in, e.g., machine learning and data science~\cite{han2015large,williams2006gaussian}, theoretical particle physics~\cite{sexton1994systematic,thron1998pade} and analysis of complex networks~\cite{benzi2020matrix,estrada2010network}; see~\cite{ubaru2018applications} for an extensive survey. In many of these applications, we have $B = f(A)$, where $A \in \R^{n \times n}$ is a large and sparse (or structured) matrix. In this case, $B$ cannot be formed explicitly and often the only feasible operations with $A$ are matrix-vector products (and sometimes linear system solves). This then translates into methods for approximating~\eqref{eq:trB} that must also rely on matrix-vector products or quadratic forms with $B$, which are, e.g., performed by applying a polyomial (or rational) Krylov subspace method or a Chebyshev expansion for approximating $f(A)\vec v$ or $\vec v^T\!f(A)\vec v$; hence the name \emph{implicit} or \emph{matrix-free} trace estimation methods.

Popular methods for this task come in two flavors: On the one hand, there are stochastic estimators, with the most important ones being the classical Hutchinson estimator as established workhorse algorithm~\cite{hutchinson1990stochastic} and the recently proposed Hutch\pp{} algorithm~\cite{MMMW21-Hutch++} and its refined variants~\cite{chen2022krylov,PerssonCortinovisKressner21} as more modern examples. These estimators are black-box methods for~\eqref{eq:trB} which do not exploit any inherent structure in $B$ and mainly rely on multiplying $B$ to randomly sampled vectors. On the other hand, when $B = f(A)$ with sparse $A$, a popular other class of methods are estimators based on \emph{probing}~\cite{Frommer2021,stathopoulos2013hierarchical,tang2012probing}, a technique which aims to reduce the number of quadratic forms that are needed for a given accuracy by exploiting structure in $A$ to carefully craft specific probing vectors instead of just randomly sampling them.

In this work, we consider the combination of probing techniques with stochastic estimators which is obtained by first finding a suitable (nonzero) pattern for the probing vectors and then filling it with random samples from an appropriate distribution. The combination of these two approaches is algorithmically quite straightforward and has already been used before by practitioners; see, e.g.,~\cite{aune2014parameter,babich2012exploring,morningstar2011improved,whyte2022optimizing} for applications in Gaussian processes and theoretical particle physics. However, a general treatment and in-depth analysis is lacking so far. We perform such an analysis in this paper. Our analysis explains and highlights many of the important properties of the approach and reveals in particular for which matrix functions $f$ and matrices $A$ large gains can not only be expected, but are actually guaranteed. Our theoretical findings are illustrated and confirmed by a variety of numerical experiments. As a by-product of our analysis, we also refine classical results on sign patterns in the entries of $f(A)$.

\subsection{Outline}
This paper is organized as follows. In Section~\ref{sec:methods}, we recall the two basic trace estimation methods, namely probing and the Hutchinson estimator. Section~\ref{sec:stochastic_probing} then treats how to combine these methods, starting with an algorithmic description and then deriving formulas for the variance of the resulting estimators as well as more sophisticated tail bounds. We show that in certain situations, the expected value of the error increases only with the square root of the matrix dimension $n$, whereas in classical (deterministic) probing the error scales linearly with $n$. We conclude this section by discussing conditions on the sign pattern in $f(A)$ under which the stochastic probing method using just a single sample is already \emph{guaranteed} to improve upon the classical probing method. 
We then brief\/ly discuss the relation of the stochastic probing approach to other established variance reduction techniques---namely the recently proposed Hutch\pp{} and XTrace algorithms based on low-rank approximation of $f(A)$---in Section~\ref{sec:hutchpp}. In Section~\ref{sec:numerical_experiments}, we present a number of numerical experiments illustrating our theoretical results and the performance of the stochastic probing method, both on academic and on real-world problems. 
Concluding remarks and topics for future research are given in Section~\ref{sec:conclusions}.

\subsection{Notation}
For a symmetric matrix $A\in\R^{n\times n}$, we denote by $\graph(A)=(V,E)$ the undirected graph associated with it, so $V=\{ 1,\dots, n \}$ and $E=\{ (i,j): i\neq j, [A]_{ij}\neq 0 \}$. The geodesic distance $\dist(i,j)$ in $\graph(A)$ is the length of the shortest path in $\graph(A)$ connecting $i$ to $j$, where $\dist(i,j)=0$ if $i=j$. 
We denote the entry at position $(i,j)$ of a matrix $M \in \R^{n \times n}$ by $[M]_{ij}$, and for $\mathcal{I} \subseteq \{1,\dots,n\}$ an index set, we denote by $[M]_{\mathcal{I}}$ the submatrix $M$ built from the rows and columns with index in $\mathcal{I}$. The spectrum $\sigma(M)$ of a square matrix $M$ is the set of all its eigenvalues. We denote by $\|\cdot\|_2$ the Euclidean vector norm (and the spectral matrix norm it induces) and by $\|\cdot\|_F$ the Frobenius matrix norm. For a random variable $X$, we denote by $\E[X]$ and $\Var[X]$ its expected value and variance, respectively.

\section{Review of basic methods}\label{sec:methods}
In this section, we review the two basic trace estimation methods which form the basis of the stochastic probing method that we discuss in Section~\ref{sec:stochastic_probing}.

\subsection{Probing methods}\label{subsec:probing}

Probing methods for trace (or diagonal) estimation were introduced in~\cite{tang2012probing} and later refined, extended and analyzed in, e.g.,~\cite{Frommer2021,laeuchli2020extending,stathopoulos2013hierarchical}. The basic idea of probing methods is based on the observation that the magnitude of the entries of most matrix functions $f(A)$ exhibits an (often exponential) decay away from the sparsity pattern of $A$; see, e.g.,~\cite{Benzi99,BenziRinelli,Demko84,Frommer18,Frommer2021,Schweitzer2022a} and the references therein. 

This observation motivates to construct \emph{probing vectors} 
\begin{equation}\label{eq:probing_vectors}
    \vec v_\ell = \sum_{i\in C_\ell} \vec e_i,\quad \ell=1,\dots,m,
\end{equation}
where $\vec e_i$ denotes the $i$th vector of the canonical basis of $\R^n$ and the sets $C_\ell, \ell = 1,\dots,m$ are a partitioning of the index set $V = \{1,\dots,n\}$. We denote their sizes by $|C_\ell| =: n_\ell$. The (deterministic) probing approximation corresponding to~\eqref{eq:probing_vectors} is then given by 
\begin{equation}\label{eq:deterministic_probing_approximation}
    \traceP(f(A))=\sum_{\ell=1}^m \vec v_\ell^Tf(A) \vec v_\ell
\end{equation}
with error
\begin{equation}\label{eq:trace_probing_error_deterministic}
\trace(f(A)) - \traceP(f(A)) = -\sum_{\ell = 1}^m \sum_{\substack{i,j \in C_\ell \\ i \neq j}} [f(A)]_{ij}.
\end{equation}
One can see from \eqref{eq:trace_probing_error_deterministic} that the approximation is accurate when the entries $[f(A)]_{ij}$ are small in modulus for $i$ and $j$ in the same index set $C_\ell$. Therefore, if the entries of $f(A)$ decay away from the sparsity pattern of $A$, a typical approach is to construct the sets $C_1,\dots,C_m$ via a distance-$d$ coloring of $\graph(A)$ according to the following definition. 

\begin{definition} \label{def:distance-d-coloring}
Let $G=(V,E)$ be an undirected graph. If the mapping $\col:V\to \{1,\dots,m\}$ is such that $\col(i)\neq \col(j)$ whenever $\dist(i,j)\leq d$, the corresponding partitioning of $V$ given via $C_\ell =\{ i\in V:\col(i)=\ell \}$, $\ell=1,\dots,m$, is called a \emph{distance-$d$ coloring} (with $m$ colors) of $G$.
\end{definition}

Note that a distance-1 coloring is the standard node coloring of a graph, where any two adjacent nodes have different colors. The computation of a distance-$d$ coloring (in particular for large values of $d$) is a computationally demanding task in general, even if one does not aim for an optimal coloring. It is common practice to use a greedy approach for obtaining a suboptimal coloring with affordable computation cost~\cite{BenziRinelliSimunec2022,Frommer2021,SchimmelThesis}, but even then, the computation of the coloring is often the most expensive part of the overall probing method; see, e.g., the discussion in~\cite[Section~6]{BenziRinelliSimunec2022}. 

For special graph structures, colorings with a small but not necessarily minimal number of colors are known in closed form and therefore available cheaply. As these will play a role for our subsequent analysis, we recall them here.
\begin{proposition}[from Section~2 in~\cite{Frommer2021}]\label{pro:coloring_special_cases}
\phantom{a}

\begin{enumerate}
\item Let $A$ be $\beta$-banded, i.e., $[A]_{ij}=0$ for $|i-j|>\beta$. Then
a distance-$d$ coloring for $\graph(A)$ with $m=d\beta + 1$ colors is given by
\begin{equation}
    \label{eq:coloring_banded}
    \col(i) = (i-1) \, \operatorname{mod} \, (d\beta + 1) + 1,\quad i=1,\dots,n,
\end{equation}
and this coloring is optimal if all entries within the band of $A$ are nonzero.
\item Let $\graph(A)=(V,E)$ be a regular $D$-dimensional lattice of size $n_1 \times \dots \times n_D$ and let its nodes be labeled as $V=\{ v\in \Z^D: 0\leq [v]_k \leq n_k-1 \text{ for  $k=1,\dots,D$}\}$. Then a distance-$d$ coloring for $\graph(A)$ with $m = (d+1)^D$ colors is given by
\begin{equation}
    \label{eq:coloring_lattice}
    \col(v)=\left( \sum_{k=0}^{D-1}\widetilde{[v]}_k (d+1)^k \right)+1,\quad \text{where}\quad \widetilde{[v]}_k=[v]_k \text{ mod }(d+1).
\end{equation}
\end{enumerate}    
\end{proposition}

\begin{remark}
\phantom{a}
\begin{enumerate}
    \item The coloring \eqref{eq:coloring_lattice} is not optimal. For $D=2$, the optimal distance-$d$ coloring is explicitly known and needs $\left\lceil \frac{1}{2}(d+1)^2 \right\rceil$ colors. The regular structure of the coloring~\eqref{eq:coloring_lattice} allows an easier analysis of corresponding probing approximations, though, as discussed in \cite{Frommer2021,SchimmelThesis}. For $D > 2$, optimal distance-$d$ colorings of the lattice are not known explicitly for general $d > 1$.
    \item A different way of coloring regular lattices is used in the \emph{hierarchical probing} approach from~\cite{laeuchli2020extending,stathopoulos2013hierarchical}, which starts with a distance-1 coloring and recursively subdivides it. This typically results in a non-optimal coloring which uses more colors than a coloring computed by a greedy method but allows to reuse certain computational results if one detects that the probing approximation is not accurate enough and the distance for the coloring thus needs to be increased.
\end{enumerate}
\end{remark}

Based on the colorings from Proposition~\ref{pro:coloring_special_cases}, one can derive bounds for the error~\eqref{eq:trace_probing_error_deterministic} under the assumption that $f(A)$ has the exponential decay property
\begin{equation}\label{eq:exponential_decay}
    |[f(A)]_{ij}| \leq cq^{\dist(i,j)} 
\end{equation}
for all $i,j$, where $c>0$, $0\leq q <1$ are constants independent of $i,j$.

\begin{theorem}[Theorems~4.1 and~4.2 in~\cite{Frommer2021}] \label{thm:probing_bound}
\phantom{a}

\begin{enumerate}
\item Let $A\in\R^{n\times n}$ be $\beta$-banded such that $f(A)$ fulfills \eqref{eq:exponential_decay} and let $\traceP_d(f(A))$ be the probing approximation \eqref{eq:deterministic_probing_approximation} of $\trace(f(A))$ associated with the distance-$d$ coloring \eqref{eq:coloring_banded}. Then
\begin{equation*}
    |\trace(f(A))-\traceP_d(f(A))|
    \leq 
    n \cdot 2c\, \frac{q^d}{1-q^d}.
\end{equation*}
\item Let $A\in\R^{n\times n}$ be such that $\graph(A)$ is a regular $D$-dimensional lattice and $f(A)$ fulfills \eqref{eq:exponential_decay}. Let $\traceP_d(f(A))$ be the probing approximation of $\trace(f(A))$ associated with the distance-$d$ coloring \eqref{eq:coloring_lattice}. Then
    \begin{equation*}
        |\trace(f(A))-\traceP_d(f(A))|
        \leq 
        n \cdot 2c D\, \Li_{1-D}(q^{d}),
    \end{equation*}
    where $ \Li_{s}(z)=\sum_{k=1}^\infty \frac{z^i}{i^s}$ is the polylogarithm.
\end{enumerate}
\end{theorem}

\begin{remark}
Note that when $D$ is a positive integer, $\Li_{1-D}(z)$ is a rational function and explicit representations are known, e.g.,
\begin{equation*}
    \Li_0(z) = \frac{z}{1-z}, \quad \Li_{-1}(z) = \frac{z}{(1-z)^2},\quad \Li_{-2}(z) = \frac{z+4z^2+z^3}{(1-z)^4}.
\end{equation*}
In all these cases $\Li_{1-D}(q^d)=\O(q^d)$ for large $d$.
\end{remark}

It is also possible to obtain an error bound that holds for general graphs $\graph(A)$ (i.e., without relying on a specific structure or the use of a specific coloring), based on polynomial approximation of $f$ over an interval $[a,b] \supset \sigma(A)$. To this purpose we introduce the notation
\begin{equation*}
    E_d(f,[a,b])=\min_{p_d\in \poly_d} \max_{x\in[a,b]}|f(x)-p_d(x)|
\end{equation*}
for the polynomial approximation error, where $\poly_d$ is the set of all polynomials with degree at most~$d$. Theorem 4.4 in \cite{Frommer2021} gives a bound for the case when $E_d(f,[a,b])$ decays exponentially. With essentially the same proof one can obtain the following general result.
\begin{theorem}\label{the:error_deterministic_probing_general}
Let $A\in\R^{n\times n}$ be symmetric with $\sigma(A)\subset[a,b]$. Let $\traceP_d(f(A))$ be the probing approximation of $\trace(f(A))$ associated with a distance-$d$ coloring of $\graph(A)$. Then
\begin{equation*}
    |\trace(f(A)) - \traceP_d(f(A))|
    \leq 
    2n \cdot E_d(f,[a,b]).
\end{equation*}
    
\end{theorem}

The numerical examples in Section~\ref{sec:numerical_experiments} and in~\cite[Section 6]{Frommer2021} illustrate that the error of the deterministic probing approximation indeed scales linearly with the size $n$, which suggests that $\O(n)$ error bounds as in Theorem~\ref{the:error_deterministic_probing_general} are the best we can achieve with this method. We will see that a better scaling is achieved via stochastic probing in Section~\ref{sec:stochastic_probing}.

\subsection{Hutchinson estimator}\label{subsec:hutchinson}
The Hutchinson estimator was first proposed in~\cite{girard1989fast} for estimating the trace of influence matrices in penalized least squares problems and in~\cite{hutchinson1990stochastic} for estimating $\trace(A^{-1})$, the trace of the inverse of a matrix. It can be applied in exactly the same way for any implicitly given matrix, though. It is based on the following fact: If $\vec x$ is a random vector with independent and identically distributed (i.i.d.) components such that $\E[\vec x]=0$ and $\E[\vec x \vec x^T]=I$, then $\E[\vec x^T B \vec x]=\trace (B)$. By sampling $N$ independent random vectors with this property, one obtains the Hutchinson estimator
\begin{equation*}
    \trace^{(H)}_N(B)=\frac{1}{N}\sum_{k=1}^N (\vec x^{(k)})^T B \vec x^{(k)}.
\end{equation*}
From the linearity of the expected value it follows that $\E[\trace^{(H)}_N(B)]=\trace(B)$. The accuracy of the approximation $\trace^{(H)}_N(B)$ depends on $N$ and the distribution of the entries of the sample vectors.
Two important examples are \emph{Rademacher vectors}, whose entries are i.i.d., taking the values $\pm1$ with equal probability $\tfrac{1}{2}$, or \emph{Gaussian vectors}, whose entries are i.i.d.\ according to the standard normal distribution. For these particular cases, the variance of the estimator is explicitly known; see, e.g.~\cite[Corollary~3.2]{FrommerMultilevel22}. For convenience we from now on use the notation $\Off(B)$ for the off-diagonal part of a matrix $B$.

\begin{theorem} \label{thm:variance_standard}
Let $B\in\R^{n\times n}$ be symmetric and let $\vec x^{(i)}, i = 1,\dots,N$ be independent random vectors with i.i.d.\ components.

\begin{enumerate}
    \item If the $\vec x^{(i)}$ are Rademacher vectors, then 
    \begin{equation*}
        \Var[\trace^{(H)}_N(B)] = \frac{2}{N}\|\Off(B)\|_F = \frac{2}{N}\sum_{\substack{i,j=1\\i\neq j}}^n|[B]_{ij}|^2.
    \end{equation*} 
    \item If the $\vec x^{(i)}$ are Gaussian vectors, then 
    \begin{equation*}
        \Var[\trace^{(H)}_N(B)]=\frac{2}{N}\|B\|_F^2 
        = \frac{2}{N}\sum_{i,j=1}^n|[B]_{ij}|^2..
    \end{equation*}
\end{enumerate}
\end{theorem}

Similar results hold for non-symmetric $B$, involving $\Off(B+B^T)$.

A more accurate analysis of the Hutchinson approximation is possible by deriving tail bounds of the following form: given a target accuracy $\varepsilon>0$ and a failure probability $\delta$, find $N$ such that the $(\varepsilon,\delta)$ approximation
\begin{equation*}
    \prob(|\trace(B)-\trace^{(H)}_N(B)|\geq \varepsilon)\leq \delta
\end{equation*}
holds. A vast literature on the topic has been developed in the past years; see for instance~\cite{avron2011randomized,CortinovisKressner2021,roosta2015improved}. The following recent result is currently the one with the tightest tail bounds. 

\begin{theorem}[Theorem~1 and Corollary~1 in~\cite{CortinovisKressner2021}]\label{the:tailbound_cortinoviskressner}
    Let $B\in\R^{n\times n}$ be symmetric and let $\vec x^{(1)}, \dots, \vec x^{(N)}$ be independent random vectors.
    \begin{enumerate}
        \item If the $\vec x^{(i)}$ are Rademacher vectors, then 
        \begin{equation*}
        \mathbb{P}(|\trace(B) - \trace^{(H)}_N(B)| \geq \varepsilon) \leq 2 \exp\left(- \frac{N\varepsilon^2}{8\|\Off(B)\|_F^2 + 8\varepsilon\|\Off(B)\|_2}\right)
    \end{equation*}
    for every $\varepsilon > 0$. In particular, for $N \geq \frac{8}{\varepsilon^2}(\|\Off(B)\|_F^2 + \varepsilon\|\Off(B)\|_2) \log\frac{2}{\delta}$ it holds that $\mathbb{P}(|\trace(B) - \trace^{(H)}_N(B)| \geq \varepsilon) \leq \delta$.
    \item  If the $\vec x^{(i)}$ are Gaussian vectors, then
    \begin{equation*}
        \prob(|\trace(B)-\trace^{(H)}_N(B)|\geq \varepsilon) \leq 2\exp\left( -\frac{N\varepsilon^2}{4\|B\|_F^2+4\varepsilon\|B\|_2} \right)
    \end{equation*}
    for all $\varepsilon>0$. In particular, for $N\geq \frac{4}{\varepsilon^2}(\|B\|^2_F+\varepsilon\|B\|_2)\log\frac{2}{\delta}$ it holds that $\prob(|\trace(B)-\trace^{(H)}_N(B)|\geq \varepsilon)\leq \delta$.
    \end{enumerate}    
\end{theorem}

 A consequence of these tail bounds is that in order to obtain an $(\varepsilon,\delta)$ approximation, we need $N=\O(\varepsilon^{-2}\log(\frac{1}{\delta}))$ samples. This asymptotic result is sharp as discussed in~\cite{wimmer2014optimal}. For favorable distributions of the eigenvalues of $B$ it is possible to obtain better asymptotic results by resorting to the more powerful Hutch\pp{} estimator~\cite{MMMW21-Hutch++}. We will come back to this point in Section~\ref{sec:hutchpp}. 

\section{Stochastic probing}\label{sec:stochastic_probing}
Stochastic probing methods combine the probing approach discussed in Section~\ref{subsec:probing} with the Hutchinson estimator from Section~\ref{subsec:hutchinson}. Stochastic probing has already been used under the name ``dilution''---restricted to distance $d=1$---for approximating the trace of the inverse in lattice quantum chromodynamics computations~\cite{babich2012exploring,morningstar2011improved} and for more general $d$ in~\cite{aune2014parameter,gambhir2017deflation}, but without a theoretical analysis.

\subsection{Description of the method}\label{subsec:stochastic_probing_method}
We assume that $C_1,\dots,C_m$ (with $|C_\ell| = n_\ell$) is a partition of $V = \{1,\dots,n\}$ associated with a distance-$d$ coloring of $\graph(A)$. We define a stochastic probing vector associated with $C_\ell$ as
\begin{equation}
    \label{eq:stochastic_probing_vectors}
    \vec w_\ell = \sum_{i\in C_\ell} X_i \vec e_i, \quad \ell = 1,\dots,m, 
\end{equation}
where $X_i$, $i\in C_\ell$, are i.i.d.\ random variables such that $\E[X_i]=0$ and $\E[X_i^2]=1$ for all $i\in C_\ell$. In the case of Rademacher variables we call $\vec w_\ell$ a \emph{Rademacher probing vector}, while in the case of Gaussian variables we call $\vec w_\ell$ a \emph{Gaussian probing vector}.   Basic properties of these stochastic probing vectors are summarized in the following proposition.
\begin{proposition} \label{prop:stochastic_probing_error}
    Let $A\in \R^{n\times n}$ be symmetric. Let a partition $C_1,\dots,C_m$ of $V = \{1,\dots,n\}$ associated with a distance-$d$ coloring of $\graph(A)$ be given and let $\vec w_\ell$ be a stochastic probing vector associated with $C_\ell$, $\ell=1,\dots,m,$ as defined in \eqref{eq:stochastic_probing_vectors}. Then
    \begin{equation}
        \label{eq:identity_quadform_generic}
        \vec w_\ell^T f(A) \vec w_\ell 
        =
        \sum_{k\in C_\ell} X_k^2[f(A)]_{kk} + \sum_{\substack{i,j\in C_\ell \\ i\neq j}} X_i X_j [f(A)]_{ij},
    \end{equation}
    and thus 
    \begin{equation}
        \label{eq:quadform_expectation}
        \E[\vec w_\ell^T f(A) \vec w_\ell] = \trace([f(A)]_{C_\ell}) = \sum_{k\in C_\ell} [f(A)]_{kk}.
    \end{equation}
    
\end{proposition}

\begin{proof}

    The identity \eqref{eq:identity_quadform_generic} follows trivially from the expression \eqref{eq:stochastic_probing_vectors} that defines $\vec w_\ell$. Since $X_i$ and $X_j$ are independent for $i\neq j$ we get $\E[X_i X_j] = \E[X_i] \E[X_j] = 0$, and since $\E[X_k^2]=1$ for all $k$, we obtain~\eqref{eq:quadform_expectation}.
\end{proof}

 In the spirit of the Hutchinson estimator, we can sample $N_\ell$ probing vectors $\vec w_\ell^{(1)},\dots, \vec w_\ell^{(N_\ell)}$ associated with $C_\ell$ and obtain the approximation
\begin{equation*}
   \frac{1}{N_\ell}\sum_{s=1}^{N_\ell} (\vec w_\ell^{(s)})^T f(A) \vec w_\ell^{(s)}
   \approx 
   \trace([f(A)]_{C_\ell})
\end{equation*}
for the ``partial trace'' belonging to color $\ell$. To approximate the full trace, we sum estimates for all the partial traces according to the following definition.

\begin{definition} \label{stochastic_probing:def} Let $C_\ell, \ell = 1,\ldots,m$, be a partitioning of $V=\{1,\ldots,n\}$ and let the number of samples for each $\ell$ be collected in the $m$-tuple $\N=(N_1,\dots,N_m)$. Then the \emph{stochastic probing estimator} of $\trace(f(A))$ is given by
\begin{equation}
    \label{eq:stochastic_probing_definition}
    \traceP_{d}^{\N}(f(A))
    :=
    \sum_{\ell = 1}^m \frac{1}{N_\ell} \sum_{s=1}^{N_\ell} (\vec w_\ell^{(s)})^T f(A) \vec w_\ell^{(s)}.
\end{equation}
\end{definition}

In view of \eqref{eq:quadform_expectation}, we have
\begin{equation*}
    \E[\traceP_d^{\N}(f(A))]
    =
    \trace(f(A)).
\end{equation*}

\subsection{Variance of the stochastic probing estimator}\label{subsec:variance_stochastic_probing}

From basic properties of the variance, we have
\begin{equation}
    \label{eq:variance_stochastic_probing_generic}
    \Var[\traceP_d^{\N}(f(A))]=\sum_{\ell=1}^m \V_\ell/N_\ell, \quad \mbox{ where } \V_\ell :=\Var[\vec w_\ell^T f(A) \vec w_\ell].
\end{equation}
The individual variances $V_\ell$ can be given explicitly.
\begin{theorem} \label{thm:variance_stochastic_probing_quadforms}
    Let $A\in \R^{n\times n}$ be symmetric, let a partition $C_1, \dots C_m$ of $V = \{1,\dots,n\}$ be given and let $\vec w_\ell$ be a stochastic probing vector associated with $C_\ell, \ell \in \{1,\dots,m\}$, as defined in~\eqref{eq:stochastic_probing_vectors}. 
\begin{enumerate}
    \item If $\vec w_\ell$ is a Rademacher probing vector, then
    \begin{equation}
        \label{eq:variance_color_Rademacher}
        \V_\ell = 2\|\Off([f(A)]_{C_\ell})\|_F^2 = 2\sum_{\substack{i,j\in C_\ell \\ i\neq j}}|[f(A)]_{ij}|^2.
    \end{equation}
    \item If $\vec w_\ell$ is a Gaussian probing vector, then
    \begin{equation}
        \label{eq:variance_color_Gaussian}
        \V_\ell = 2\|[f(A)]_{C_\ell}\|_F^2 =
        2\sum_{i,j\in C_\ell}|[f(A)]_{ij}|^2.
    \end{equation}
\end{enumerate}
    
\end{theorem}

\begin{proof}
For any $\ell = 1,\dots,m$ we have $\vec w_\ell^T f(A) \vec w_\ell = ([\vec w_\ell]_{C_\ell})^T [f(A)]_{C_\ell} [\vec w_\ell]_{C_\ell}$. In view of Theorem~\ref{thm:variance_standard}, if $\vec w_\ell$ is a Rademacher probing vector we get
    \begin{equation*}
        \Var[\vec w_\ell^T f(A) \vec w_\ell] = \Var[([\vec w_\ell]_{C_\ell})^T [f(A)]_{C_\ell} [\vec w_\ell]_{C_\ell}] = \| \Off([f(A)]_{C_\ell})\|. 
    \end{equation*}
    Similarly, if $\vec w_\ell$ is a Gaussian probing vector, then
    \begin{equation*}
        \Var[\vec w_\ell^T f(A) \vec w_\ell] = \Var[([\vec w_\ell]_{C_\ell})^T [f(A)]_{C_\ell} [\vec w_\ell]_{C_\ell}] = \| [f(A)]_{C_\ell}\|_F. 
    \end{equation*}
    This concludes the proof. 
\end{proof}

\begin{remark}
The diagonal entries of $f(A)$ are present in the summation in~\eqref{eq:variance_color_Gaussian}, but not in~\eqref{eq:variance_color_Rademacher}. This suggests---and is actually confirmed by numerical experiments not reported here---that if a decay away from the sparsity pattern of $A$ is present in $f(A)$, then the variance is much smaller in the Rademacher case, since the diagonal entries will dominate in~\eqref{eq:variance_color_Gaussian}.  From now on, we will therefore only consider the Rademacher distribution for our analysis and experiments.
\end{remark}

From~\eqref{eq:variance_stochastic_probing_generic} and Theorem~\ref{thm:variance_stochastic_probing_quadforms}, we obtain the variance of the Rademacher probing approximation. We immediately state it for distance-$d$ colorings, although it also holds for a general partitioning $C_1, \dots C_m$.

\begin{lemma}
    \label{lem:variance_Rademacher_probing}
    Let $A\in \R^{n\times n}$ be symmetric and let $C_\ell, \ell = 1,\ldots,m,$ be a distance-$d$ coloring of $\graph(A)$. Further, let $\N=(N_1,\dots,N_m)$, let $\vec w_\ell^{(1)},\dots,\vec w_\ell^{(N_\ell)}$ be Rademacher probing vectors associated with $C_\ell$, and let $\traceP^{\N}_d(f(A))$ be the corresponding Rademacher probing approximation~\eqref{eq:stochastic_probing_definition}. Then, with $V_\ell$ from \eqref{eq:variance_color_Rademacher},
    \begin{equation*}
        \Var[\traceP_d^{\N}(f(A))]
        =
        \sum_{\ell=1}^m\frac{\V_\ell}{N_\ell}
        =
        2\sum_{\ell=1}^m \frac{1}{N_\ell} \|\Off([f(A)]_{C_\ell})\|_F^2.
    \end{equation*}
  
\end{lemma}

\begin{remark}
Lemma~\ref{lem:variance_Rademacher_probing} tells us that the variance becomes much smaller as $d$ increases if a decay on the entries is present. However, as in the deterministic case with the formula~\eqref{eq:trace_probing_error_deterministic}, the exact value of $V_\ell = \|\Off([f(A)]_{C_\ell})\|_F^2$ on the right hand side is not known, since $f(A)$ is not explicitly available. Bounds on $V_\ell$ for special cases will be given in Section~\ref{subsec:variance_specific}.
\end{remark}

Assuming that a distance-$d$ coloring of $\graph(A)$ is already given, the computational cost of the stochastic probing method is proportional to the number of matrix-vector products (or rather quadratic forms) that need to be evaluated. In the following, we therefore only count the number of quadratic forms to gauge the efficiency of the method. For example, the cost for computing the stochastic probing estimator $\traceP^{\N}_d(f(A))$ covered in Lemma~\ref{lem:variance_Rademacher_probing} is taken to be $N_1+\dots +N_m$ quadratic forms with $f(A)$. 

It is in general not advisable to use the same number of samples for each $C_\ell$: When the partition comes from a distance-$d$ coloring of $\graph(A)$, it is often the case that the variances $\V_\ell$ vary widely. This is, e.g., typically the case if some sets $C_\ell$ in the coloring are much smaller than the others. A cost optimal approach will thus use different sample sizes for different colors, as we develop now. 

Suppose that we aim to obtain an estimator with overall variance at most $\varepsilon^2$. Then, taking into account the expression~\eqref{eq:variance_stochastic_probing_generic} for the variance and assuming that we know the individual variances $\V_\ell$, the overall lowest number of quadratic forms is obtained by solving the optimization problem 
\begin{equation}
    \label{eq:minimization_variance}
        \min\limits_{N_1,\dots,N_m \in \mathbb{N}} \quad N_1+\dots+N_m  \qquad
        \text{s.t. }  \sum_{\ell=1}^m \V_\ell/N_\ell = \varepsilon^2.
\end{equation}
While the discrete problem~\eqref{eq:minimization_variance} is difficult to solve, it is easy to do so if one relaxes it by allowing $N_\ell\geq 0$ to be real. Similar optimization problems occur in general multi-level Monte Carlo methods and have been solved before; see, e.g.,~\cite{FrommerMultilevel22,giles2015multilevel,hallman2022multilevel}. The solution of the relaxed version of~\eqref{eq:minimization_variance} is given by 
\begin{equation}
    N_\ell = \mu \sqrt{\V_\ell}, \quad \mu = \varepsilon^{-2}\sum_{\ell=1}^m \sqrt{\V_\ell}. \label{eq:minimization_variance_solution}
\end{equation}
To obtain integer numbers of samples, one can simply pick the ceiling of the values, $\lceil N_\ell\rceil$, in \eqref{eq:minimization_variance_solution}. 

As discussed before, the variances $\V_\ell$ required for computing~\eqref{eq:minimization_variance_solution} are generally unknown in practice. In some cases it is possible to bound them a priori---as we will discuss in Section~\ref{subsec:variance_specific}---or they can be estimated on-the-fly during the computation---as is, e.g., described in~\cite{hallman2022multilevel} in the context of stochastic multilevel trace estimation methods.

\subsection{Tail bounds for the stochastic probing estimator}\label{subsec:tailbounds}

According to the central limit theorem, for large values of the $N_\ell$, the stochastic estimator $\traceP_{d}^{\N}(f(A))$ for the trace is approximately normally distributed with mean $\trace(f(A))$ and standard deviation  $\sigma = \left( \sum_{\ell=1}^m V_\ell / N_\ell \right)^{1/2}$. This gives the approximate tail bound
\[
    \mathbb P\left(  |\trace(f(A)) - \traceP_{d}^{\N}(f(A))| \geq \varepsilon  \right) \lesssim  1- \erf\left( \frac{\varepsilon}{\sqrt{2}\sigma}  \right),
\]
with the error function $\erf(z) =  \tfrac{2}{\sqrt{\pi}} \int_0^z e^{-t^2} dt$.
Solving for $\sigma^2$ results in
\begin{equation} \label{eq:approx_tail_bound}
 \mathbb P\left(  |\trace(f(A)) - \traceP_{d}^{\N}(f(A))| \geq \varepsilon  \right) \leq \delta \quad \mbox{ if } \enspace
  \sum_{\ell=1}^m V_\ell/N_\ell  \, \lesssim \, \frac{\varepsilon^2}{2 (\inverf(1-\delta))^2}.
\end{equation}

The relation \eqref{eq:approx_tail_bound} is not entirely satisfactory, since it does not quantify how ``inexact'' the inequality is and how large the $N_\ell$ should be. It has the advantage, though, that it applies for any probability distribution used for the stochastic probing vectors. For Rademacher vectors we now derive tail bounds for the stochastic probing estimator which do not rely on the central limit theorem and which hold for any choice of the sample sizes $N_\ell$. To this end we recall the following result from~\cite{CortinovisKressner2021} which, actually, is at the basis of Theorem~\ref{the:tailbound_cortinoviskressner}.
\begin{theorem}[Theorem 2 in \cite{CortinovisKressner2021}]\label{thm:tailbound_zerodiag} Let $\vec x$ be a Rademacher vector of length $n$ and let $M\in \mathbb{R}^{n\times n}$ be a nonzero matrix such that $[M]_{ii}=0$ for $i=1,\dots,n$. Then, for all $\varepsilon>0$,
    \begin{equation*}
        \mathbb P\left( |\vec x^T M \vec x|\geq \varepsilon  \right) \leq 2\exp \left(-\frac{\varepsilon^2}{8\|M\|_F^2+8\varepsilon\|M\|_2}\right).
    \end{equation*}
\end{theorem}

Using this theorem we can derive an analogue of Theorem~\ref{the:tailbound_cortinoviskressner} for the stochastic probing method.

\begin{theorem}
    Let $A$ be symmetric and let $f$ and $A$ be such that $f(A)$ is nonzero. Let a distance-$d$ coloring $C_1,\dots,C_m$ be given, let $\N=(N_1,\dots,N_m)$ and let $\vec w_\ell^{(s)}$ be Rademacher probing vectors, with $\traceP_{d}^{\N}(f(A))$ the corresponding trace estimate~\eqref{eq:stochastic_probing_definition}. Then
    \begin{equation*}
        \mathbb{P}\left(|\trace(f(A)) - \traceP_{d}^{\N}(f(A)\right)| \geq \varepsilon) 
        \leq 
        2 \exp\left(- \frac{\varepsilon^2}{8\eta_1 + 8\varepsilon \eta_2}\right)
    \end{equation*}
    for every $\varepsilon > 0$, where 
    \begin{equation}\label{eq:M1M2}
        \eta_1 := \sum_{\ell=1}^m \frac{1}{N_\ell}\|  \Off([f(A)]_{C_\ell}) \|_F^2 = \sum_{\ell =1}^m \V_\ell/N_\ell,
        \qquad
        \eta_2 := \max_{\ell=1,\dots,m}\frac{1}{N_\ell}\|\Off([f(A)]_{C_\ell}) \|_2,
    \end{equation}
    with $\V_\ell$ the variances from \eqref{eq:variance_color_Rademacher}.
    In particular, if we choose $N_1,\dots,N_m$ such that 
    \begin{equation}\label{eq:tailbound_probing_M12}
        \eta_1+\varepsilon \eta_2\leq \frac{\varepsilon^2}{8\log \frac{2}{\delta}},
    \end{equation}
    we have $\mathbb{P}\left(|\trace(f(A)) - \traceP_{d}^{\N}(f(A)\right)| \geq \varepsilon) \leq \delta$.
\end{theorem}

\begin{proof}
     Define the block diagonal matrices
    \begin{equation*}
        M_\ell 
        :=
        \frac{1}{N_\ell}
        \left[
            \begin{array}{ccc}
               \Off([f(A)]_{C_\ell})  \\
               & \ddots \\
                &&  \Off([f(A)]_{C_\ell})
            \end{array}
        \right]\in \mathbb{R}^{n_\ell N_\ell \times n_\ell N_\ell}, \enspace \ell=1,\dots,m,
    \end{equation*}
    and
    \begin{equation*}
        M := 
        \left[
            \begin{array}{ccc}
              M_1 \\ 
               & \ddots \\
                &&  M_m 
            \end{array}
        \right]\in \mathbb{R}^{\widehat{N}\times \widehat{N}},
    \end{equation*}
    where $\widehat{N}=\sum_{\ell=1}^m n_\ell N_\ell.$ Then, $\traceP_d^{\N}(f(A))-\trace(f(A))=\vec x^TM\vec x$ with $\vec x$  a Rademacher vector of length $\widehat{N}$ of the form 
    \begin{equation*}
        \vec x = \left[ 
        \begin{array}{c}
             \widehat{\vec w}_1  \\
             \vdots \\
              \widehat{\vec w}_m
        \end{array}
        \right]\in\mathbb{R}^{\widehat{N}}, 
        \quad \text{ where }
        \widehat{\vec w}_\ell = \left[ 
        \begin{array}{c}
            [\vec w_\ell^{(1)}]_{C_\ell}   \\
             \vdots \\{}
            [\vec w_\ell^{(N_\ell)}]_{C_\ell} 
        \end{array}
        \right]\in \mathbb R^{n_\ell N_\ell}, \ell = 1,\dots,m,
    \end{equation*}
    and each $[\vec w_\ell^{(s)}]_{C_\ell}\in \mathbb R^{n_\ell}$ is a Rademacher vector of length $n_\ell$. We conclude by applying Theorem~\ref{thm:tailbound_zerodiag}, noting that $\eta_1=\|M\|_F^2$ and $\eta_2 = \|M\|_2$.
\end{proof}

Similarly to the discussion at the end of Section~\ref{subsec:variance_stochastic_probing}, we now elaborate on how to choose the number of samples $N_\ell, \ell = 1,\dots,m,$ in order to obtain a cost-efficient $(\varepsilon,\delta)$ approximation of the trace, based on inequality~\eqref{eq:tailbound_probing_M12}. To minimize the number $N_1+\dots +N_\ell$ of quadratic forms involved in the computation of \eqref{eq:stochastic_probing_definition}, one can solve
\begin{align}
    \begin{aligned}
    \min\limits_{N_1,\dots,N_m \in \mathbb{N}}  N_1+\dots+ N_m &\qquad
    \text{s.t.} &\enspace \eta_1+\varepsilon \eta_2\leq \varepsilon^2/8\log(2/\delta),
    \end{aligned} \label{eq:minimization_tailbound}
\end{align}
where $\eta_1, \eta_2$ are defined in~\eqref{eq:M1M2}. A way to solve this approximately is to ignore the term $\varepsilon \eta_2$, due to the small factor $\varepsilon$ and since we expect $\varepsilon \eta_2\ll \eta_1$; a similar reasoning in a slightly different context is also used in~\cite{PerssonCortinovisKressner21}. With this assumption, noting also that $\eta_1=\Var[\traceP_d^{\N}(f(A))]$, the problem~\eqref{eq:minimization_tailbound} reduces to
\begin{align*}
    \begin{aligned}
    \min\limits_{N_1,\dots,N_m \in \mathbb{N}} N_1+\dots+ N_m &\qquad
    \text{s.t.} &\enspace \Var[\traceP_d^{\N}(f(A))] = \varepsilon^2/8\log(2/\delta).
    \end{aligned}
\end{align*}
This is the same as \eqref{eq:minimization_variance}, where $\varepsilon^2$ is replaced by $\varepsilon^2/8\log(2/\delta)$. Hence, by readapting \eqref{eq:minimization_variance_solution} to this case, we find
\begin{equation*}
    N_\ell = \lceil \mu \sqrt{\V_\ell}\rceil, \quad \mu = 8\varepsilon^{-2}\log \frac{2}{\delta} \sum_{\ell=1}^m \sqrt{\V_\ell}, \quad \V_\ell = \|\Off([f(A)]_{C_\ell})\|_F^2. 
\end{equation*}

\subsection{A priori bounds for the variance in specific cases}\label{subsec:variance_specific}

We now discuss three specific situations in which we can give a priori bounds on the variances $V_\ell$ and thus, according to Lemma~\ref{lem:variance_Rademacher_probing}, on the variance of the stochastic probing estimator $\traceP^{\N}_d(f(A))$.  These bounds can then, in turn, be used to determine suitable numbers $N_\ell$ of samples based on the tail bounds of Section~\ref{subsec:tailbounds}. An important feature of all three bounds is that they depend {\em only linearly} on the dimensions $n_\ell$. Interpreting the square root of the variance as a measure for the accuracy of the Rademacher probing method, we thus have a {\em sublinear} dependence, proportional to the square root of the dimension. This is an ``order $\tfrac{1}{2}$'' improvement over the non-stochastic probing method, where the accuracy depends linearly on the dimension; see Theorem~\ref{thm:probing_bound}. 

\begin{proposition}\label{prop:bounds_Vl_banded_grid}
  Let $A\in\R^{n\times n}$ be symmetric, let $\graph(A)$ be its graph, and assume that $f(A)$ has the exponential decay property~\eqref{eq:exponential_decay} with constants $c > 0$ and $0 < q < 1$.
\begin{itemize} 
    \item[(i)] If $\graph(A)$ is $\beta$-banded and $\traceP^{\N}_d(f(A))$ is the Rademacher probing approximation~\eqref{eq:stochastic_probing_definition} associated with the distance-$d$ coloring \eqref{eq:coloring_banded}, then  the variances $V_\ell$ from \eqref{eq:variance_color_Rademacher} satisfy
    \begin{equation}
    \label{eq:bound_variance_color_banded}
        \V_\ell\leq n_\ell \cdot 4 c^2\frac{q^{2d}}{1-q^{2d}}, 
    \end{equation} 
    \item[(ii)] If $\graph(A)$ is a regular $D$-dimensional lattice and $\traceP^{\N}_d(f(A))$ is the Rademacher probing approximation \eqref{eq:stochastic_probing_definition} associated with the distance-$d$ coloring \eqref{eq:coloring_lattice},  
    then the variances $V_\ell$ from \eqref{eq:variance_color_Rademacher} satisfy
    \begin{equation}
    \label{eq:bound_variance_color_lattice}
        \V_\ell\leq n_\ell \cdot 4c^2D\Li_{1-D}(q^{2d}). 
    \end{equation}
\end{itemize}
\end{proposition}

\begin{proof}
    Let $C_1,\ldots,C_m$ be the colors of the coloring in (i) or (ii). From the exponential decay property \eqref{eq:exponential_decay}, we obtain 
    \begin{equation*}
        \V_\ell = 2\sum_{\substack{i,j\in C_\ell \\ i\neq j}} |[f(A)]_{ij}|^2 \leq 2 \sum_{\substack{i,j\in C_\ell \\ i\neq j}} c^2 q^{2\dist(i,j)}.
    \end{equation*}

   In the case of (i), the distance-$d$ coloring \eqref{eq:coloring_banded} has the property that for given $i$ all other nodes in $C_\ell$ have a distance from $i$ which is a multiple $\gamma d$ of $d$, and there are at most two such nodes for each $\gamma \in \mathbb{N}$; cf.~\cite[Section~4]{Frommer2021} for a detailed discussion. This is why for each $i$ we can bound $\sum_{j \in C_\ell, j\neq i} q^{2\dist(i,j)} \leq 2 \sum_{\gamma=1}^\infty q^{2d\gamma} = 2 q^{2d}/(1-q^{2d})$, which gives \eqref{eq:bound_variance_color_banded}. 

   In the case of (ii) and the coloring \eqref{eq:coloring_lattice}, for a given node $i \in C_\ell$ it was shown in \cite[Section~4]{Frommer2021} that all other nodes in $C_\ell$ have again a distance which is a multiple $\gamma d$ of $d$, and for each $\gamma$ there are this time at most $2D\gamma^{D-1}$ nodes which have this distance. 
    Thus
    \[
    \V_\ell = 2 \sum_{i \in C_\ell} \sum_{j \in C_\ell, j\neq i} |[f(A)]_{ij}|^2 \leq 2n_\ell \sum_{\gamma=1}^\infty  2D\gamma^{D-1} \cdot c^2 q^{2d\gamma} =  n_\ell \cdot 4 c^2\Li_{1-D}(q^{2d}),
    \]
   which is \eqref{eq:bound_variance_color_lattice}. 
\end{proof}

The third situation that we consider is when $[f(A)]_{ij}$ has constant sign for $\dist(i,j) > d$. Here, we can establish a bound using the polynomial approximation error $E_d(f,[a,b])$.

\begin{proposition}
\label{prop:variance_constant_sign}
Let $A\in \R^{n\times n}$ be symmetric with $\sigma(A)\subset[a,b]$ and let $C_1,\dots,C_m$ be a distance-$d$ coloring of $\graph(A)$. Let $\traceP^{\N}_d(f(A))$ be the Rademacher probing approximation~\eqref{eq:stochastic_probing_definition} associated with this coloring. Suppose that $[f(A)]_{ij}$ has constant sign for all $i,j$ with $\dist(i,j)>d$. Then
    \begin{equation}
    \label{eq:bound_variance_color_polyapprox}
        \V_\ell \leq n_\ell \cdot  4 \big(E_d(f,[a,b])\big)^2.
    \end{equation}
\end{proposition}
\begin{proof}
Denote by $\vec v_\ell$ the deterministic probing vector associated with the color $C_\ell$, $\ell=1,\dots, m$, so that we have
\begin{equation*}
\label{eq:partial_probing_error_deterministic}
    \trace([f(A)]_{C_\ell})-\vec v_\ell^T f(A) \vec v_\ell
    = 
    -\sum_{\substack{i, j \in C_\ell \\ i \neq j}} [f(A)]_{ij}.
\end{equation*}
By assumption, all terms on the right hand side have the same sign,
which gives
\[
\sum_{\substack{i,j\in C_\ell \\ i\neq j}}|[f(A)]_{ij}| = 
\big| \sum_{\substack{i,j\in C_\ell \\ i\neq j}}[f(A)]_{ij} \big| = 
\left| \trace([f(A)]_{C_\ell})-\vec v_\ell^T f(A) \vec v_\ell \right|,
\]
and thus
    \begin{align}
        \nonumber
        \V_\ell = 2 \sum_{\substack{i,j\in C_\ell\\ i\neq j}}|[f(A)]_{ij}|^2
        &\leq 
        2\max_{\substack{i,j\in C_\ell\\ i\neq j}} |[f(A)]_{ij}| \sum_{\substack{i,j\in C_\ell \\ i\neq j}}|[f(A)]_{ij}|  \\
        & \leq 
        2 \max_{\dist(i,j)>d} |[f(A)]_{ij}| \cdot \left| 
        \trace([f(A)]_{C_\ell})- \vec v_\ell^T f(A) \vec v_\ell \right|.
         \label{eq:bound_variance_color_nonnegative}
    \end{align}
The first factor in \eqref{eq:bound_variance_color_nonnegative} is bounded by a quantity that does not depend on $\ell$,
\begin{equation*}
    \max_{\dist(i,j)>d} |[f(A)]_{ij}|\leq E_d(f,[a,b]);
\end{equation*}
see, e.g.,~\cite[Section~2.1]{Frommer18}. For the second term, let $p_d$ denote the polynomial of best approximation of degree $d$. Then $\trace(p_d(A)) = \vec v_\ell^T p_d(A) v_\ell$, since $[p_d(A)]_{ij} = 0$ if $\dist(i,j) > d$, and thus 
\begin{eqnarray*}
   \left| \trace([f(A)]_{C_\ell}) - \vec v_\ell^T f(A) \vec v_\ell \right| &\leq&  \left| \trace([f(A)]_{C_\ell}) - \trace([p_d(A)]_{C_\ell})\right| + \left|\trace([p_d(A)]_{C_\ell}) - \vec v_\ell^T f(A) \vec v_\ell \right| \\
   &\leq&  \sum_{k \in C_\ell} \left| [f(A)]_{kk}- [p_d(A)]_{kk})\right| + \left| \vec v_\ell^T(p_d(A)-f(A))v_\ell \right|. 
\end{eqnarray*}
Herein, each term in the sum is bounded by $E_d(f,[a,b])$, and for the second term we have 
\[
\left| \vec v_\ell^T(p_d(A)-f(A))\vec v_\ell \right| \leq \|\vec v_\ell\|^2 \cdot E_d(f,[a,b]) = n_\ell \cdot E_d(f,[a,b]).
\]
Using this in \eqref{eq:bound_variance_color_nonnegative} gives \eqref{eq:bound_variance_color_polyapprox}. 
\end{proof}
\begin{remark}\label{rem:Nl_proportionalto_sqrtnl}
When given a fixed budget $N$ of quadratic forms that one wants to invest for trace estimation via stochastic probing, our results lead to a simple heuristic for distributing these across the different colors $C_\ell$:
According to~\eqref{eq:minimization_variance_solution}, the optimal number of samples should be chosen proportional to $\sqrt{V_\ell}$, and under the assumptions of Proposition \ref{prop:bounds_Vl_banded_grid} or Proposition \ref{prop:variance_constant_sign} we further know that $V_\ell$ essentially scales as $\mathcal{O}(n_\ell)$. Thus, it is sensible to choose the numbers of samples as $N_\ell=\operatorname{round}(\nu \sqrt{n_\ell})$, where
\begin{equation*} \label{eq:optimal_Nl}
    \nu=N/\sum_{\ell=1}^m\sqrt{n_\ell},
\end{equation*}
 and $\operatorname{round}(\nu\sqrt{n_\ell})$ denotes the nearest integer to $\nu\sqrt{n_\ell}$. 
\end{remark}

We conclude this section by characterizing  classes of functions and matrices for which $[f(A)]_{ij}$ has a fixed sign for $\dist(i,j) \geq d$, i.e., situations in which Proposition~\ref{prop:variance_constant_sign} applies. We first state a classical result from~\cite{Fiedler83} which treats the case  $d=0$, i.e.\ it gives conditions such that $f(A) \geq 0$ for all $i$ and $j$. Recall that a matrix $A \in\mathbb{R}^{n \times n}$ is called a (possibly singular)  {\em M-matrix} if it can be written as $A = \theta I - B$, where all entries in $B$ are non-negative, $B \geq 0,$ and the spectral radius satisfies $\rho(B) \leq \theta$. If $A$ (and thus $B$) is symmetric, $\rho(B) = \|B\|_2 = \lambda_{\max}(B)$, the largest eigenvalue of $B$; see, e.g.\ \cite{BermanPlemmonsBook}.
\begin{lemma}[\cite{Fiedler83}] \label{lemma:nonnegative_matrix_function}
    Let  $A=\theta I-B\in \R^{n\times n}$ with $B \geq 0$ and $\|B\|_2 \leq \theta$ be a symmetric M-matrix. Suppose that $f(x)$ is continuous over $[\theta-\|B\|_2,\theta+\|B\|_2]$, analytic over $(\theta-\|B\|_2,\theta+\|B\|_2+\varepsilon)$, for some $\varepsilon>0$, and that
    \begin{equation}
        (-1)^k f^{(k)}(\theta) \geq 0\quad \text{for all $k\geq 0$}.  \label{eqn:alternating_derivatives}
    \end{equation}
    Then $f(A)\geq 0$.
\end{lemma}

This gives rise to the following proposition where we assume that the derivatives of $f$ alternate their sign  only beyond the $d$th derivative.  

\begin{proposition} \label{prop:constant_sign_constrained}
Let $A$ and $f$ be as in Lemma \ref{lemma:nonnegative_matrix_function}, except that condition \eqref{eqn:alternating_derivatives} is replaced by
\begin{equation*}
    (-1)^kf^{(k)}(\theta)\geq 0\quad \text{for $k\geq d$},
\end{equation*}
where $d$ is a positive integer. Then
\begin{equation*}
    [f(A)]_{ij}\geq 0 \quad \text{for $\dist(i,j)\geq d$}.
\end{equation*}
\end{proposition}

\begin{proof}
    We consider the Taylor series centered at $\theta$, writing it in terms of $x \in (\theta-\|B\|,\theta+\|B\|)$. Then
    \begin{equation*}
        f(x) = \sum_{k=0}^{\infty} c_k x^k=p_{d-1}(x)+g(x),
    \end{equation*}
    where $p_{d-1}(x)=\sum_{k=0}^{d-1} c_k x^k$ and $g(x)=\sum_{k=d}^{\infty} c_k x^k$.
    Since $[A^k]_{ij}=0$ for $k< \dist(i,j)$, we have that $[p_{d-1}(A)]_{ij}=0$ if $d \leq \dist(i,j)$. Moreover, 
    \begin{equation*}
        g^{(k)}(\theta) = 
        \begin{cases}
            0\quad &\text{if $k< d$}\\
            f^{(k)}(\theta)\quad &\text{if $k\geq d$}
        \end{cases},
    \end{equation*}
    so from Lemma \ref{lemma:nonnegative_matrix_function} we get that $g(A)\geq 0$. Finally,
    \begin{equation*}
        [f(A)]_{ij}=[p_{d-1}(A)]_{ij}+[g(A)]_{ij}=[g(A)]_{ij}\geq 0
    \end{equation*}
    for $\dist(i,j) \geq d$.
\end{proof}

Of course, if the assumptions on $f$ in Proposition~\ref{prop:constant_sign_constrained} hold for $-f$ instead of $f$, we obtain $[f(A)]_{ij} \leq 0 $ for $\dist(i,j) \geq d$.  

\begin{remark}\label{rem:situations_with_constant_sign}
Lemma~\ref{lemma:nonnegative_matrix_function} or Proposition~\ref{prop:constant_sign_constrained} can be applied for the following functions $f$ of a symmetric M-matrix $A$:
\begin{enumerate}
    \item Completely monotone functions, i.e., function whose derivatives satisfy $(-1)^kf^{(k)}(x)\geq 0, k = 0,1,\dots$ for all $x\in(0,\infty)$. In that case, Lemma~\ref{lemma:nonnegative_matrix_function} implies that $f(A)\geq 0$. 
    \item Bernstein functions, i.e., nonnegative functions $f$ on $(0,\infty)$ with completely monotone derivative $f^\prime$. By applying Proposition \ref{prop:constant_sign_constrained} to $-f$ we get $[f(A)]_{ij}\leq 0$ for $\dist(i,j)\geq 1$, i.e. for $i\neq j$. 
    \item The entropy function $f(x)=-x\log x$. In fact, for $\theta > 0$ we have  $f^{(k)}(\theta)=(-1)^{k-1}\theta^{1-k}(k-2)!$ for $k\geq 2$, so if $A = \theta I -B$ with $B \geq 0$ is an M-matrix, by applying Proposition \ref{prop:constant_sign_constrained} to $-f$ we get $[f(A)]_{ij}\leq 0$ for $\dist(i,j)\geq 2$. If $\theta\leq \exp(-1)$ we also get $f'(\theta)=-\log \theta -1\geq 0$, so that in this case we even have $[f(A)]_{ij}\leq 0$ for $\dist(i,j)\geq 1$, i.e. for $i\neq j$. 
    \item Fractional powers with exponent larger than $1$, i.e., $f(x)=x^\alpha$ with $\alpha >1$. If $d=\lceil \alpha\rceil$, we get
    \begin{align*}
        \begin{cases}
            f^{(k)}(x)&\geq 0 \quad \text{for $k<d$},\\
            (-1)^{k+d}f^{(k)}(x) &\geq 0\quad \text{for $k\geq d$}.
        \end{cases}
    \end{align*}
    Hence, if $d$ is even, we get $[f(A)]_{ij}\geq 0$ for $\dist(i,j)\geq d$, and if $d$ is odd we get $[f(A)]_{ij}\leq 0$ for $\dist(i,j)\geq d$.
    \item $f(x) = x\exp(-x)$, for which $f^{(k)}(x) = (-1)^{k} (x-k)\exp(-x)$. For $d = \lceil\theta\rceil$, we get that $(-1)^{k+d+1}f^{(k)}(\theta)\geq 0$ for $k\geq d$. Therefore, if $d$ is even, we get $[f(A)]_{ij}\leq 0$ for $\dist(i,j)\geq d$, and if $d$ is odd we get $[f(A)]_{ij}\geq 0$ for $\dist(i,j)\geq d$.
\end{enumerate}
\end{remark}

In all situations just outlined we also have that stochastic probing with only one sample per color is never worse than deterministic probing (which has the same computational cost).

\begin{proposition} \label{prop:stochastic_determinsitic_comp} Assume that $A$ and $f$ satisfy the assumptions of Proposition~\ref{prop:variance_constant_sign} with $C_1,\dots,C_m$ a distance-$d$ coloring of $\graph(A)$. Denote by $\traceP_d^{\mathbbm{1}}(f(A))$ the stochastic probing approximation of the trace using just one Rademacher vector for each color, i.e.\ $\N = (1,\ldots,1)$ and by $\traceP_d(f(A))$ the deterministic probing approximation defined in \eqref{eq:deterministic_probing_approximation} with respect to the same distance-$d$ coloring. Then
\[
| \traceP_d^{\mathbbm{1}}(f(A)) - \trace(f(A)) | \leq | \traceP_d(f(A)) - \trace(f(A)) |.
\]
\end{proposition}
\begin{proof}
We use the notation and the result from Proposition~\ref{prop:stochastic_probing_error} to obtain 
    \begin{eqnarray*}
 \big| \traceP_d^{\mathbbm{1}}(f(A)) - \trace(f(A)) \big|  =  \big| \sum_{\substack{i,j\in C_\ell \\ i\neq j}} X_i X_j [f(A)]_{ij} \big| \leq  \sum_{\substack{i,j\in C_\ell \\ i\neq j}} | [f(A)]_{ij}|, 
    \end{eqnarray*}
and the assumptions of Proposition~\ref{prop:variance_constant_sign} together with \eqref{eq:trace_probing_error_deterministic} to conclude 
\[
\sum_{\substack{i,j\in C_\ell \\ i\neq j}} | [f(A)]_{ij}| = \big| \sum_{\substack{i,j\in C_\ell \\ i\neq j}}  [f(A)]_{ij} \big| = | \traceP_d(f(A)) - \trace(f(A)) |.
\] 
\end{proof}
Note that this result assumes constant sign in $[f(A)]_{ij}$ only for $\dist(i,j) \geq d$. For special cases and $d=1$ the result has been observed before, for example in \cite[Section~2.3]{aune2014parameter} for the specific case of the log-determinant of certain precision matrices.  

\section{Other variance reduction techniques}\label{sec:hutchpp}
Starting with the Hutch\pp{} method from \cite{MMMW21-Hutch++}, several new variance reduction techniques have been suggested recently.  
Hutch\pp{} is essentially based on the observation that $\trace(B)$ might be well approximated by the trace of a low-rank approximation of $B$ if its eigenvalues decay rapidly,
and this low-rank approximation $B_{\text{lr}}$ is obtained stochastically via one step of a block power iteration with random starting vectors. The trace is then computed as the sum of $\trace(B_{\text{lr}})$, which is obtained directly, and $\trace(B-B_{\text{lr}})$. For the latter term one uses the Hutchinson estimator which now has smaller variance. 

The XTrace estimator \cite{XTrace} is a modification of Hutch\pp{} which uses resampling and implements the {\em exchangeability criterion}, a symmetry which is not present in Hutch\pp{}, but which is necessary for an estimator with optimal variance.  For both, Hutch\pp{} and XTrace, there is a modification in which the low rank approximation is replaced by the Nystr\"om approximation \cite{Nystrom}, resulting in the methods Nystr\"om\pp{} and XNysTrace; see  \cite{PerssonCortinovisKressner21} and \cite{XTrace}.

Without going into more details, we state the following theorem from \cite{XTrace} comparing the variances of the respective trace estimators in the particularly favorable case where the eigenvalues of $B$ decay exponentially. 

\begin{theorem}
    Let $B$ be positive definite and assume that its  eigenvalues satisfy $\lambda_i(B)\leq \alpha^i$ for $i=1,2,\dots$, where $\alpha\in(0,1)$. Denote by \rm{$\trace_N^{\texttt{Method}}$} the estimator using Method \rm{\texttt{Method}} investing $N$ matrix-vector multiplications. Then
    \begin{align*}
        (\Var[\trace_N^{\rm{\texttt{Hutch++}}}(B)])^{1/2} &\leq C_1(\alpha)\alpha^{N/3}; \\
        (\Var[\trace_N^{\rm{\texttt{XTrace}}}(B)])^{1/2} &\leq C_2(\alpha)\alpha^{N/2}; \\
        (\Var[\trace_N^{\rm{\texttt{XNysTrace}}}(B)])^{1/2} &\leq C_3(\alpha)\alpha^{N},
    \end{align*}
    where $C_j(\alpha)$, $j=1,2,3$, depend only on $\alpha$.
\end{theorem}

One might think that the stochastic probing method that we discussed in this paper could be further enhanced by using one of these more advanced methods instead of standard Hutchinson on each of the colors $C_\ell$ where we estimate $\trace([f(A)]_{C_\ell})$. This, however, turns out to be counter-productive, since these methods then invest too much effort in computing the low rank approximation for each color block. What we would need for the advanced estimators to be efficient is that for each color $C_\ell$ the block $[f(A)]_{C_\ell}$ can be well approximated by a matrix with small rank compared to the block size $n_\ell$. The coloring approach aims at making the off-diagonal entries of $[f(A)]_{C_\ell}$ small, so if the diagonal elements are not comparably small, there cannot be a low rank approximation to $[f(A)]_{C_\ell}$. And even if there is a rapid decay in the eigenvalues of $f(A)$, we can at best expect to have a similar decay in the eigenvalues of {\em each} color block $[f(A)]_{C_\ell}, l =1,\ldots,m$. But this means that, 
compared to without probing, we have an effort which is $m$ times as large to invest before we retrieve a good low rank approximation for each of the color blocks.

\section{Numerical experiments}\label{sec:numerical_experiments}
We now illustrate our theoretical analysis with several numerical examples. All experiments are done in MATLAB R2020b on a computer with Intel(R) Core(TM) i7-7700HQ CPU and 16 GB RAM.

\subsection{Scaling with $n$ for random geometric graphs}\label{subsec:exp_scaling}
We start with a series of experiments illustrating the results in Section~\ref{subsec:variance_specific} on the more favorable scaling of the error of stochastic vs.\ deterministic probing. 
The matrices used are Laplacians $L$ or adjacency matrices $A$ of random geometric graphs with $n$ nodes. These were obtained via the command \texttt{random\_geometric\_graph(n,radius)} from the package \texttt{NetworkX} \cite{NetworkX} in Python with $n$ ranging from $200$ to $5,\!000$ in steps of $200$. 
In order to keep similar properties for all the matrices, in particular to keep the number $m$ of colors of the distance-$d$ coloring similar for all $n$, we used $\texttt{radius} = \sqrt{\frac{\log n}{\pi n}}$ for all $n$, motivated by the results in \cite[Section 6]{PenroseRGGraphs}.

\begin{figure}
    \includegraphics{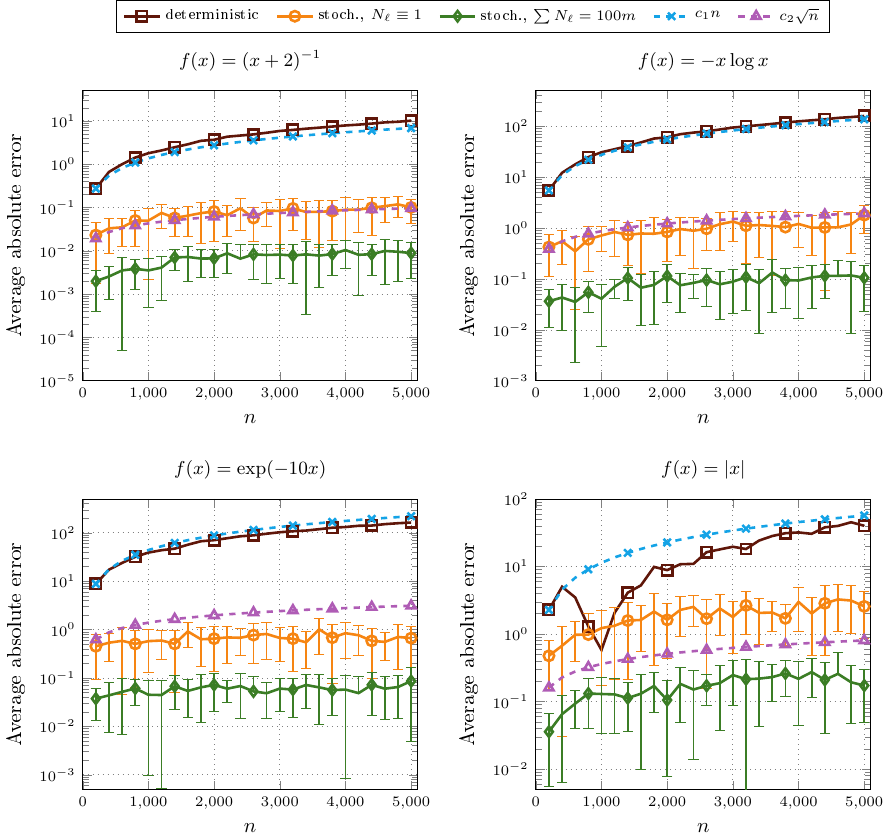}
    \caption{Errors and asymptotic behavior for the three methods in \eqref{eq:three_methods} applied to functions of the Laplacian or adjacency matrix of a random geometric graph for increasing $n$.}
    \label{fig:error_scaling_graphlap}
\end{figure}

We test four different matrix functions: the shifted inverse $f_1(L)=(L+2I)^{-1}$, the entropy function $f_2(L) = -L\log L$,  the exponential $f_3(L) = \exp(-10L)$, where $L$ is the graph Laplacian, and the absolute value $f_4(A)=|A|$ of the adjacency matrix $A$. The trace of $|A|$ is the \emph{graph energy}; cf.~\cite{GraphEnergy}.
Graph Laplacians are M-matrices, since they can be written as $dI-B$ with $d$ the maximum degree in the graph and $B \geq 0$ with $\rho(B) = d$; cf.~\cite[Chapter 6]{BermanPlemmonsBook}. Hence, in view of Remark \ref{rem:situations_with_constant_sign} the assumptions of Proposition~\ref{prop:variance_constant_sign} are satisfied for the functions $f_1(L)$ and $ f_3(L)$ for all distance-$d$ colorings and for $f_2(L)$ if $d \geq 2 $. On the other hand, we have no information on the signs in $f_4(A)$. 

In this experiment, we use a distance-3 coloring for all graphs, which is computed with a simple greedy method. With $f$ a generic symbol for one of the functions $f_1,\ldots,f_4$ and $M^{(n)}$ a generic symbol for the Laplacian or the adjacency matrix of the random geometric graph with $n$ nodes, Figure~\ref{fig:error_scaling_graphlap} reports the quantities 
\begin{equation}    \label{eq:three_methods}
\begin{array}{ll}
  |\trace(f(M^{(n)}))-\traceP_d(f(M^{(n)}))| & \enspace \mbox{(deterministic probing)}, \\
  |\trace(f(M^{(n)}))-\traceP_d^{\mathbbm{1}}(f(M^{(n)}))|    & \enspace \mbox{(stochastic probing, one vector per color)} \\
  |\trace(f(M^{(n)}))-\traceP_d^{\N}(f(M^{(n)}))|  & \enspace \mbox{(stochastic probing with $\N = (N_1,\ldots,N_m)$, where} \\
  & \enspace  \; \mbox{$N_\ell=\nu \sqrt{|C_\ell|}$ and $\sum_{\ell=1}^m N_\ell=100m$).}
\end{array}
\end{equation}
The choice of $N_{\ell}$ in the third case is motivated by Remark \ref{rem:Nl_proportionalto_sqrtnl}, and we expect the error of $\traceP_d^{\N}(f(M^{(n)}))$ to be smaller by one order of magnitude as compared to the error of $\traceP_d^{\mathbbm{1}}(f(M^{(n)}))$. For the stochastic methods, we actually average these errors over 20 runs and also report the sample standard deviation in the plots.

We also display the functions $g_1(n)= c_1 n$ and $g_2(n) = c_2\sqrt{n}$ with suitably chosen constants $c_1$ and $c_2$ which allow to easily appreciate the expected scaling behavior. 

Figure~\ref{fig:error_scaling_graphlap} shows that stochastic probing with just one vector per color performs indeed better---and this by one to two orders of magnitude---than deterministic probing. Moreover, the expected decrease by one order of magnitude for stochastic probing with $N_\ell$ samples per color is clearly visible in all four examples. In addition, for $f_1,f_2$ and $f_4$ we see that the error of stochastic probing (with one vector) scales with the square root of the dimension whereas it scales linearly with the dimension for deterministic probing. This is exactly the scaling of the bounds on the error (or the variance) that we obtained in our analysis in Section~\ref{subsec:variance_specific}, and this analysis applies to $f_1,f_2$ and $f_3$. For $f_1$ and $f_2$ the actual error is very close to the bounds, whereas for $f_3$ the bounds are more pessimistic: the accuracy of stochastic probing appears to almost not depend on the dimension, whereas the growth of the error for deterministic probing appears to be less than linear. For the graph energy, i.e., function $f_4$, the analysis of Section~\ref{subsec:variance_specific} cannot be applied. Interestingly, though, we observe again linear growth of the error for deterministic probing and growth with the square root of the dimension for stochastic probing. 

\subsection{Scaling with the distance $d$}
\begin{figure}
\includegraphics{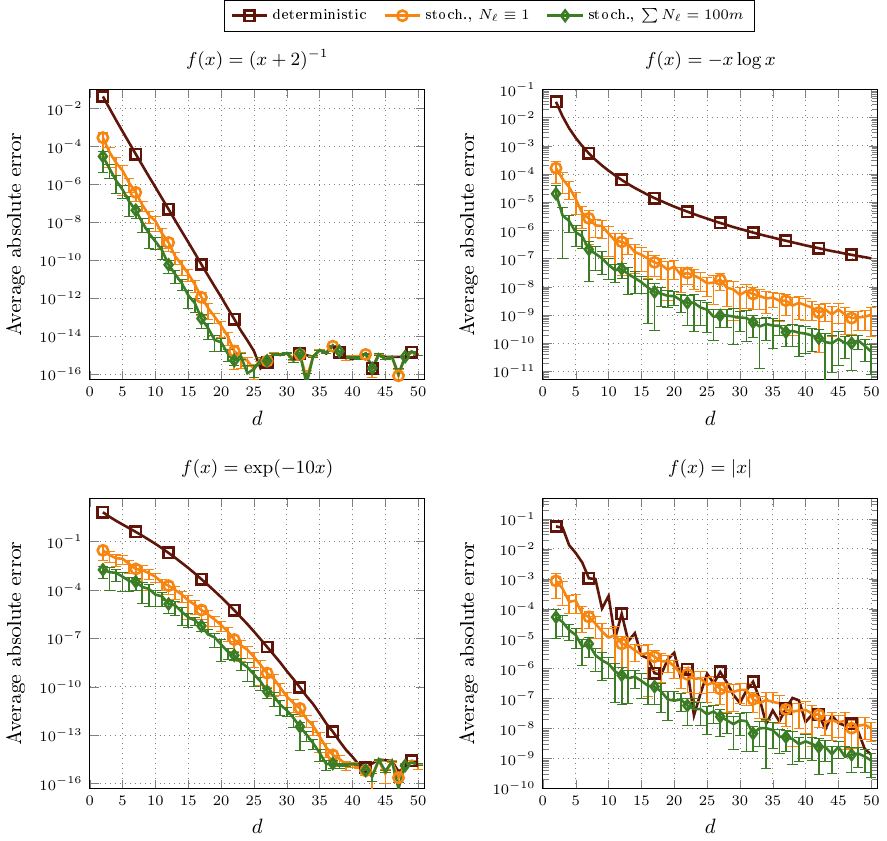}
\caption{Errors and asymptotic behavior for the three methods in \eqref{eq:three_methods} applied to the Laplacian or adjacency matrix of the District of Columbia road network for increasing $d$. }\label{fig:scaling_distance}
\end{figure}
In the next experiment we analyze the scaling with the distance $d$ of the coloring, while keeping the problem size fixed. As a test graph we consider the road network \texttt{DC} of the District of Columbia\footnote{available at \url{http://www.diag.uniroma1.it/challenge9/data/tiger/}},
more specifically its biggest connected component which has $n=9,\!522$ nodes. Road networks typically exhibit a large world structure, hence we expect to need only few colors in a distance-$d$ coloring. As before we consider the test functions $(L+2I)^{-1}$, $-L\log L$, $\exp(-10L)$ and $|A|$, where $L$ and $A$ are the graph Laplacian and the adjacency matrix, respectively.
We again compare deterministic probing, stochastic probing with just one vector per color (i.e.\ $m$ vectors in total) and stochastic probing with $\sum_{\ell=1}^m N_\ell = 100m$ with each $N_\ell$ proportional to $\sqrt{|C_\ell|}$ as motivated in Remark~\ref{rem:Nl_proportionalto_sqrtnl}.  This time we report the average \emph{relative} error and variance over $20$ runs of the stochastic methods; cf.\  Figure \ref{fig:scaling_distance}. 

As expected in view of the results in Section~\ref{sec:stochastic_probing}, for the first three functions the error for stochastic probing scales similarly as a function of $d$ as for deterministic probing  while being smaller by a few orders of magnitude. In the fourth example the error in stochastic probing is smaller for the smaller values of $d$ but becomes increasingly closer to the error of deterministic probing for larger values of~$d$.

\subsection{Comparison with other variance reduction techniques}
In a last series of experiments we consider the other variance reduction techniques outlined in Section~\ref{sec:hutchpp}. Since these techniques use low-rank approximations of $f(A)$, their performance depends on the eigenvalue distribution in $f(A)$. We compare these methods with plain Hutchinson and with stochastic probing with distances distances $d=1,3,5$. The test matrices are $f_1(L)=(L+2I)^{-1}$, $f_2(\widetilde{L})=-\widetilde{L}\log \widetilde{L}$ and $f_3(L)=\exp(-10L)$, where $L$ is again the Laplacian $L$ of the 
% \texttt{minnesota} 
\texttt{DC}
graph and $\widetilde{L}=L/\trace(L)$ is the Laplacian normalized to have unit trace. 
This normalization is common in the context of the graph entropy \cite{BenziRinelliSimunec2022,Braunstein} and ensures that $-\widetilde{L}\log \widetilde{L}$ is positive semidefinite, so that all methods can be applied. The colorings for $d=1,3,5$ computed with the greedy procedure described in~\cite{BenziRinelliSimunec2022,Frommer2021,SchimmelThesis} have $m=4,15,35$ colors, respectively. 

\begin{figure}
\includegraphics{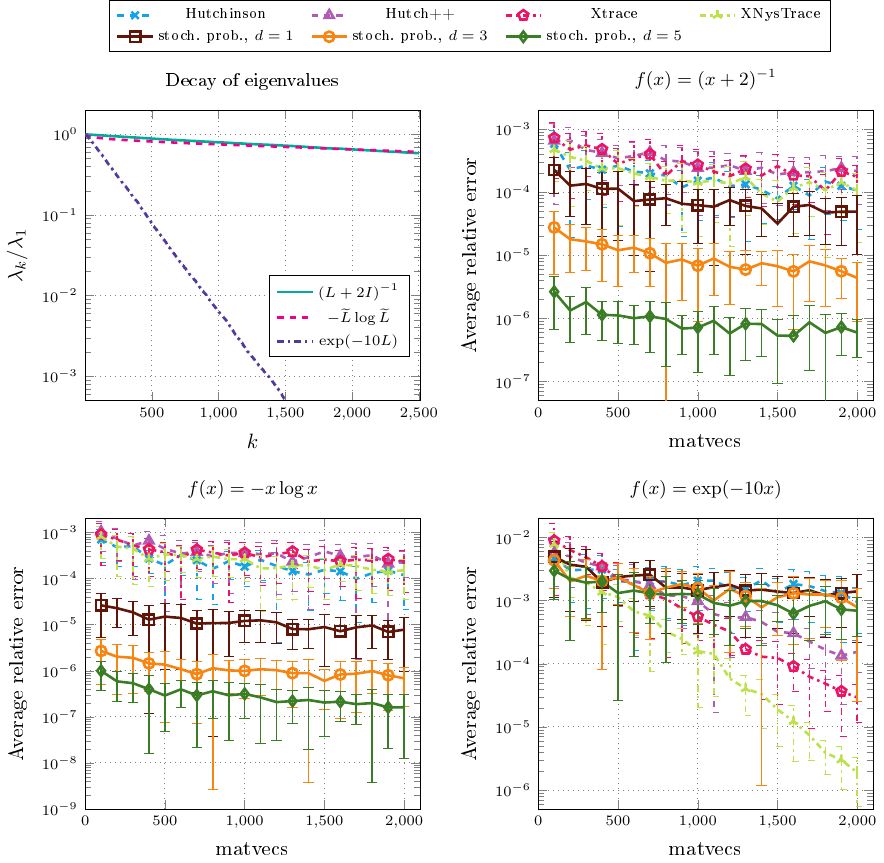}

\caption{Comparison of the accuracy of different trace estimators applied to the Laplacian of the District of Columbia road network as a function of cost (matvecs $f(L)\vec x$). \label{fig:comparison_estimators_matvec}}
\end{figure}

The largest $2,\!500$ eigenvalues of all three test matrices are depicted in the top left plot of Figure~\ref{fig:comparison_estimators_matvec}. The remaining three plots report the error of the different trace estimation methods as a function of the total number of matrix-vector multiplications. Since these evaluations are by far the most costly component in any of the methods, their number quite accurately reflects the computational cost. The essence of Figure~\ref{fig:comparison_estimators_matvec} is that for situations where the eigenvalues do not show a pronounced decay (functions $f_1$ and $f_2$), the techniques using low rank approximations perform significantly worse than stochastic probing. The situation is reversed, though, when the eigenvalues decay rapidly as in $f_3(L)$. The plots for the shifted inverse and the entropy also illustrate that using a larger distance in the coloring gives more accurate results in stochastic probing for the same cost, an observation which should not come as a big surprise.

\section{Conclusions \& Outlook}\label{sec:conclusions}
We analyzed the performance of stochastic probing methods for the trace estimation of functions of sparse matrices. We derived formulas for the variance and tail bounds of this combination between the stochastic Hutchinson's estimator and the probing estimators based on distance-$d$ colorings. For some common cases, for instance when the matrix argument is banded, when the associated graph is a regular grid, or when the entries $[f(A)]_{ij}$ have constant sign for $\dist(i,j)\geq d$, we derived bounds on the variance showing that the error scales on average with the square root of the size, in contrast to the linear scaling exhibited by the deterministic method. Our theory is validated by several numerical experiments where we observed the scaling of the error with the size and compared the performance with other known estimators, indicating when stochastic probing can be the method of choice. Further developments could consist in finding more efficient ways to compute the coloring, depending on the structure of the graph associated with the matrix argument.

\section*{Acknowledgment} 
We would like to thank Alice Cortinovis and Daniel Kressner for inspiring and fruitful discussions on the topic.

\bibliographystyle{amsplain}
\bibliography{biblio}

\end{document}